\documentclass[12pt]{amsart}
\usepackage{latexsym, amsmath, amssymb,amsthm,amsopn,amsfonts}
\usepackage{version}
\usepackage{epsfig,graphics,color,graphicx,graphpap}
\usepackage{amssymb}
\newcommand{\redsout}{\bgroup\markoverwith{\textcolor{red}{\rule[0.5ex]{2pt}{.4pt}}}\ULon}
\newcommand{\LC}{\left(}
\newcommand{\RC}{\right)}

\newcommand{\p}{\partial}

\numberwithin{equation}{section}
\newtheorem{theorem}{Theorem}[section]
\newtheorem{corollary}[theorem]{Corollary}
\newtheorem{proposition}[theorem]{Proposition}
\newtheorem{lemma}[theorem]{Lemma}
\newtheorem{definition}{Definition}[section]

\newtheorem{remark}{Remark}[section]
 
\newcommand{\R}{\mathbb R}
\author[Lai]{Ru-Yu Lai}
\address{School of Mathematics, University of Minnesota, Minneapolis, MN 55455, USA}
\curraddr{}
\email{rylai@umn.edu }
\author[Lin]{Yi-Hsuan Lin}
\address{Department of Mathematics, University of Washington, Seattle, WA 98195, USA}
\curraddr{}
\email{yihsuanlin3@gmail.com}
 
\title[Global uniqueness for the fractional semilinear Schr\"{o}dinger equation]{Global uniqueness for the fractional semilinear Schr\"{o}dinger equation}
 
\begin{document}

\maketitle
\begin{abstract}
We study global uniqueness in an inverse problem for the fractional semilinear Schr\"{o}dinger equation $(-\Delta)^{s}u+q(x,u)=0$ with $s\in (0,1)$. We show that an unknown function $q(x,u)$ can be uniquely determined by the Cauchy data set. 
In particular, this result holds for any space dimension greater than or equal to $2$. Moreover, we demonstrate the comparison principle and provide a $L^\infty$ estimate for this nonlocal equation under appropriate regularity assumptions. % that extend earlier works. 
\end{abstract}

\section{Introduction}

Let $\Omega$ be a bounded domain in $\mathbb{R}^{n},\ n\geq2$ with Lipschitz
boundary $\partial\Omega$.
We study the nonlocal type inverse problem for the fractional semilinear
Schr\"{o}dinger equation with the exterior Dirichlet data 
\begin{equation}
\left\{
\begin{array}{rl}
(-\Delta)^{s}u+q(x,u)=0 & \mbox{ in }\Omega,\\
u=g & \mbox{ in }\Omega_{e},
\end{array}\right.\label{Dirichlet Problem}
\end{equation}
where $s\in(0,1)$, $g\in C^3_0 (\Omega_e)$, and
\begin{equation*}
\Omega_{e}:=\mathbb{R}^{n}\backslash\overline{\Omega}
\end{equation*}
is the exterior domain of $\Omega$. Here the fractional Laplacian $(-\Delta)^{s}$ is defined by 
\[
(-\Delta)^{s}u=c_{n,s}\mathrm{P.V.}\int_{\mathbb{R}^{n}}\dfrac{u(x)-u(y)}{|x-y|^{n+2s}}dy,
\]
for $u\in H^s(\mathbb R^n)$, where P.V. is the principal value and  
\begin{equation}
c_{n,s}=\frac{\Gamma(\frac{n}{2}+s)}{|\Gamma(-s)|}\frac{4^{s}}{\pi^{n/2}}\label{c(n,s) constant}
\end{equation}
is a constant that was explicitly calculated in \cite{di2012hitchhiks}.

The study of fractional nonlinear Schr\"odinger (FNS) equations arises in the investigation of the quantum effects in Bose-Einstein Condensation \cite{Uzar}. %Especially, the fractional Laplacian have applications in finance, and mathematical biology[REFs]. 
In ideal boson systems, the classical Gross-Pitaevskii (GP) equations can describe condensation of weakly interacting boson atoms at a low temperature where the probability density of quantum particles is conserved. However in the inhomogeneous media with long-range (nonlocal) interactions between particles, this yields the density profile no longer retains its shape as in the classical GP equations. This dynamics is described by the fractional GP equations, known as the FNS equation, in which the turbulence and decoherence emerge. It was observed in \cite{Kay} that the turbulence appears from the nonlocal property of the fractional Laplacian; while the local nonlinearity helps maintain coherence of the density profile.

To study the equation \eqref{Dirichlet Problem}, we assume that the function $q(x,t):\overline\Omega\times\mathbb{R}\rightarrow \mathbb{R}$ fulfilling the following conditions:
\begin{equation}
q(x,t)\mbox{ and }\p_t q(x,t)\mbox{ are continuous for }(x,t)\in\overline{\Omega}\times\mathbb{R}.\label{Condition 1}
\end{equation} Moreover, suppose that there exist constants $\mu>0$ and $\delta\in(2, (2n-2s)/(n-2s))$
such that 
\begin{equation} 
\begin{cases}
|q(x,t)|\leq\mu(1+|t|^{\delta-1})\mbox{ for all }(x,t)\in\Omega\times\mathbb{R},\\
\lim_{t\to0}\dfrac{q(x,t)}{t}=0\mbox{ uniformly in }x\in\Omega,%\\
%tq(x,t)\geq 0\mbox{ for all }(x,t)\in\Omega\times\mathbb{R}.%
\end{cases}\label{Condition 2}
\end{equation}
%Consider the function 
\[
%Q(x,t):=\int_{0}^{t}q(x,\tau)d\tau,
\]
and there exist constants $b_0 \in (0,1)$ and 
$r>0$ such that 
%\begin{equation}
	%0<a_{0}Q(x,t)\leq tq(x,t),\label{Condition 3}
%\end{equation}
\begin{align}\label{Condition 4}
0<{q(x,t)\over t}\leq b_0 \p_t q (x,t),\mbox{ for any }x\in\overline{\Omega},\ |t|\geq r.
\end{align}
Meanwhile, we further assume that there is a constant $0<M_0<\infty $ such that 
\begin{equation}\label{Condition for linearized equation}
0 \leq  \partial_tq(x,t)\leq M_0,\mbox{ for any }(x,t) \in \overline\Omega \times \mathbb R.
\end{equation}
The condition \eqref{Condition for linearized equation} will be  utilized to characterize the well-posedness for the linearized
equation of $(-\Delta)^{s}u+q(x,u)=0$. For the nonlinear equation \eqref{Dirichlet Problem}, 
the weak solution $u\in H^{s}(\mathbb{R}^{n})$ exists provided that
the coefficient $q(x,t)$ satisfies \eqref{Condition 1}-\eqref{Condition 4} is discussed in Section \ref{Section 2}. % for more details for the existence of the nonlinear equation.
However, very little is known in general about uniqueness of the weak solution $u$ of \eqref{Dirichlet Problem}. We would like to point out that the uniqueness up to translations of the nontrivial solution of the fractional nonlinear equation holds for certain nonlinearity $q(x,u)$, we refer to  \cite{frank2016uniqueness} and references therein. %[add refs, Frank, Lenzmann, Silvestre]. 

We consider the nonlocal inverse problem with related nonlocal Cauchy data set, instead of the Dirichlet to Neumann (DN) map, $\Lambda_q: u|_{\Omega_e}\to (-\Delta)^su|_{\Omega_e}$ defined in \cite{ghosh2016calder}, due to the lack of uniqueness of solutions for \eqref{Dirichlet Problem}. The Cauchy data set is defined by
\begin{align*} 
\mathcal{C}_{q}^{\Omega_e}=\left\{ (u|_{\Omega_{e}},\mathcal{N}_{q}^su|_{\Omega_{e}}):\ u\in H^s(\R^n) \hbox{ is a solution of \eqref{Dirichlet Problem}  }\right\},
\end{align*}
where 
\[
\mathcal{N}_{q}^s u(x):=c_{n,s}\int_{\Omega}\dfrac{u(x)-u(y)}{|x-y|^{n+2s}}dy
\]
stands for the nonlocal Neumann derivative and the constant $c_{n,s}$
is the same as \eqref{c(n,s) constant}. Note that when the equation \eqref{Dirichlet Problem} has a unique solution, the inverse problem is to recover $q(x,u)$ from the DN map. %in \cite{ghosh2016calder}. 
For more details about the
nonlocal Neumann derivative $\mathcal{N}_{q}^s$, DN map $\Lambda_{q}$, and their connection, we refer to \eqref{relation of Ns and DN map} and 
\cite{dipierro2014nonlocal,ghosh2016calder}.

In this paper, we focus on the Calder\'{o}n problem for the fractional semilinear Schr\"{o}dinger equation, 
that is, to recover the coefficient $q(x,u)$ from the collected external data set $\mathcal{C}_q^{\Omega_e}$. It is the nonlinear and nonlocal analogue of well-known Calder\'on problem, the mathematical model of electrical impedance tomography. As a noninvasive type of medical imaging, the electrical conductivity of the object is inferred from voltage and current measurements collected only on the surface of the object. 
There are several aspects in the Calder\'on problem, including uniqueness, stability estimates, reconstruction, and numerical algorithms for the known conductivity. 
For the classical semilinear Schr\"{o}dinger equation $-\Delta u +q(x,u)=0$, global uniqueness of an inverse problem with the related DN map on full boundary $\p\Omega$ is due to \cite{ isakov1994global, sun2010inverse} when $n\geq 3$ and to \cite{VictorN, sun2010inverse} when $n=2$. We refer to the survey paper \cite{uhlmann2009calderon} for recent developments in inverse boundary value problems for linear and nonlinear elliptic equations. 
The uniqueness result has been studied for fractional Schr\"{o}dinger equation in \cite{ghosh2016calder} and for variable coefficients nonlocal elliptic operators in \cite{ghosh2017calder}. The stability estimate for the fractional Schr\"{o}dinger equation was shown in \cite{ruland2017fractional}. 

%In the current paper, we consider the global uniqueness for nonlinear fractional equations. The result stated in Theorem \ref{MAIN THEOREM} generalizes the global uniqueness result of \cite{ghosh2016calder} for the case where $q(x,u)=q(x)$.

%\textit{\color{blue} In this paper, we are interesting in It is natural to ask that the In this paper, we are interesting in property will hold under the model of the semilinear fractional Schr\"{o}dinger equation or not. The main effort of this paper is to give a positive answer to this question. 
For each coefficient $q(x,u)$, we define a set $\mathcal A_q\subset \mathbb R^n \times \mathbb R$ by 
\begin{equation}\notag
\mathcal A_q:=\{(x,u)\in\Omega \times \mathbb R:\mbox{  there exists a solution }u=u(x)\mbox{ of } \eqref{Dirichlet Problem}\}.
\end{equation}
%Now, we are ready to answer the question by demonstrating our main theorem in this article as follows.

The main result in this paper is stated in Theorem \ref{MAIN THEOREM}, which solves the uniqueness for the fractional semilinear inverse problem with partial data 
%the partial data nonlocal inverse problem for the fractional semilinear Schr\"{o}dinger inverse problem 
for arbitrary dimension $n\geq 2$. Theorem \ref{MAIN THEOREM} generalizes global uniqueness result of \cite{ghosh2016calder} for the case where $q(x,u)=q(x)$.
\begin{theorem}
\label{MAIN THEOREM} Let $\Omega\subset \mathbb R^n $ be a bounded Lipschitz domain and $0<s<1$. Suppose that $q_{1}(x,t)$ and $q_{2}(x,t)$ satisfy the conditions \eqref{Condition 1}-\eqref{Condition for linearized equation}. Let $W\subset \Omega_e$ be an arbitrary open set. Suppose that the partial Cauchy data sets $\mathcal{C}_{q_1}^{W}=\mathcal{C}_{q_2}^{W}$, that is, 
\begin{equation}\label{Partial Cauchy data}
\left\{ (u_1|_W,\mathcal{N}_{q_1}^su_1|_W)\right\}=\left\{ (u_2|_W	,\mathcal{N}_{q_2}^su_2|_W) \right\},
\end{equation}
where $u_j$ are solutions to $(-\Delta)^su_j+q_j(x,u_j)=0$ in $\Omega$  with $u_j=g$ in $\Omega_e$ for $j=1,2$, for any $g\in C^3_0 (W)$.
Then $\mathcal{A}_{q_1}=\mathcal{A}_{q_2}$ and 
\begin{align*}
q_{1}(x,u(x))=q_{2}(x,u(x)) \mbox{ in }\mathcal A_{q_1}.
\end{align*}
\end{theorem}
\begin{remark}
Different from the Calder\'on problem for the semilinear (local) elliptic
equations (see \cite{isakov1994global,sun2010inverse}), without assuming 
\begin{equation}\label{q12}
  q_1(x,0)=q_2(x,0)=0 
\end{equation}
in Theorem \ref{MAIN THEOREM}, global uniqueness for the fractional semilinear Schr\"{o}dinger equation still holds.
\end{remark}

The proof of Theorem \ref{MAIN THEOREM} starts by showing the comparison principle for the linear fractional Schr\"{o}dinger equation and the $L^\infty$ estimate for the related solutions. The estimate plays a crucial role in the linearization argument that reduces the inverse problem for the fractional nonlinear equation to the inverse problem for the fractional linear equation. Notice that there exists a strong uniqueness property for the fractional Laplacian, that is, for any $u$ in $\R^n$ satisfies $u|_W=(-\Delta)^su|_W=0$ in some open set $W$, then $u$ is identically zero in $\R^n$ (see \cite[Theorem 1.2]{ghosh2016calder}). Applying this strong uniqueness property, we derive global uniqueness result without the condition \eqref{q12}.  % the uniqueness result for the linear fractional Schr\"odinger equation in \cite{ghosh2016calder}

The paper is organized as follows. In  section \ref{Section 2}, we
introduce fundamental tools of the fractional semilinear equation. 
The comparison principle and the estimate for solutions are discussed in section 3.
In section 4, we prove Theorem \ref{MAIN THEOREM} by utilizing the linearization scheme. % and use global uniqueness of the fractional Schr\"odinger equation (see \cite{ghosh2016calder}) .

\section{Preliminaries\label{Section 2}}

First, let us begin with the fractional Sobolev spaces. Let $0<s<1$,
$H^{s}(\mathbb{R}^{n})=W^{s,2}(\mathbb{R}^{n})$ is the $L^{2}$-based
fractional Sobolev space with norm 
\[
\|u\|_{H^{s}(\mathbb{R}^{n})}=\|u\|_{L^{2}(\mathbb{R}^{n})}+\|(-\Delta)^{s/2}u\|_{L^{2}(\mathbb{R}^{n})}.
\]
Let $O\subset\mathbb{R}^{n}$ be an open set (not necessarily bounded),
then we define 
\[
H^{s}(O)=\{u|_{O}:\  u\in H^{s}(\mathbb{R}^{n})\}
\]
and 
\[
H_{0}^{s}(O)=\mbox{closure of }C_{c}^{\infty}(O)\mbox{ in }H^{s}(O).
\]

Next, we characterize the existence of solutions to our main nonlocal
problem \eqref{Dirichlet Problem}. 
\begin{lemma}
\label{Lemma existence}(Existence of weak solutions) Let $s\in(0,1)$ and $q(x,t)$
be a scalar-valued function satisfying \eqref{Condition 1}-\eqref{Condition 4}. For any $g\in C^3_{0}(\Omega_{e})$,
there exists at least one solution $u\in H^{s}(\mathbb{R}^{n})$ of
the nonlocal Dirichlet problem \eqref{Dirichlet Problem}.\end{lemma}
\begin{proof}
For any $g\in C^3_0(\Omega_{e})$, define the function $\widetilde{g}$
to be an extension of $g$ by 
\[
\widetilde{g}:=\begin{cases}
0 & \mbox{ in }\Omega,\\
g & \mbox{ in }\Omega_{e}.
\end{cases}
\]
It is easy to see that $\widetilde{g}\in C^2_0(\mathbb{R}^{n})\subset H^2(\mathbb R^n)$.
Now, consider a function $w:=u-\widetilde{g}$ and we have the fact
$q(x,u(x))=q(x,w(x))$ for $x\in\Omega$ since $\widetilde{g}=0$
in $\Omega$. Then we can rewrite the equation \eqref{Dirichlet Problem}
as 
\begin{equation}
\left\{ 
\begin{array}{rl}
(-\Delta)^{s}w+q(x,w)+h(x)=0 & \mbox{ in }\Omega,\\
w=0 & \mbox{ in }\Omega_{e},
\end{array}\right.\label{Nonlocal problem with zero boundary}
\end{equation}
where $h(x)=(-\Delta)^{s}\widetilde{g}(x)\in L^{2}(\mathbb{R}^{n})$
(see \cite[Remark 2.2]{ghosh2016calder}). Therefore, by using \cite[Theorem 11.2]{bisci2016variational},
one can see that there exists a weak solution $w\in H_{0}^{s}(\Omega)$
to the equation \eqref{Nonlocal problem with zero boundary}. This
implies that there exists a solution $u=w+\widetilde{g}\in H^{s}(\mathbb{R}^{n})$
of \eqref{Dirichlet Problem} such that $u=g$ in $\Omega_{e}$. \end{proof}
\begin{remark}
In fact, from \cite[Theorem 11.2]{bisci2016variational}, there exist
infinitely many weak solutions $\{w_{k}\}_{k\in\mathbb{N}}\subset H_{0}^{s}(\Omega)$
of \eqref{Nonlocal problem with zero boundary} such that 
\[
\iint_{\mathbb{R}^{n}\times\mathbb{R}^{n}}\dfrac{|w_{k}(x)-w_{k}(z)|^{2}}{|x-z|^{n+2s}}dxdz\to\infty\mbox{ as }k\to\infty.
\]
We only choose $w_{\ell}\in H_{0}^{s}(\Omega)$ for some $\ell\in\mathbb{N}$
to be a weak solution of \eqref{Nonlocal problem with zero boundary}.
Hence, for this semilinear nonlocal problem, it is more natural to
formulate the Calder\'on problem by using the characterization of the Cauchy
data set.
\end{remark}
In this paper, our exterior Dirichlet data $g$ are given in $C^{3}_0(\Omega_e) \subset H^{2s}(\mathbb R^n)$, so that $(-\Delta)^s u\in H^{-s}(\mathbb R^n)$, where $u\in H^s(\mathbb R^n)$ is a solution of \eqref{Dirichlet Problem}. By using the relation 
\begin{align}\label{relation of Ns and DN map}
(-\Delta)^s u|_{\Omega_e}=\mathcal N_q^su|_{\Omega_e}-m u|_{\Omega_e} +(-\Delta)^s(E_0g)|_{\Omega_e}
\end{align} 
(see Lemma 3.2 in \cite{ghosh2016calder}), where $m\in C^\infty (\Omega_e)$ is defined by $m(x)=c_{n,s}\int _\Omega {1\over |x-y|^{n+2s}}dy$ and $E_0$ is a zero extension in $\Omega$ such that $E_0g(x)=g(x)$ for $x\in \Omega_e$, $E_0g(x)=0$ for $x\in \Omega$.  
Therefore, $\mathcal N^s_qu\in H^{-s}(\mathbb R^n)$ and the Cauchy data ${(u|_{\Omega_e},\mathcal N^s_qu|_{\Omega_e})}$ can be regarded as in the function space $H^s(\Omega_e)\times H^{-s}(\Omega_e)$ (indeed, $u|_{\Omega_e}\in H^{2s}(\Omega_e)\subset H^s (\Omega_e)$).

\section{$L^{\infty}$-estimate of weak solutions}

In this section, we offer a $L^{\infty}$-estimate for the solution of the fractional Schr\"odinger equation under suitable regularity assumptions. This estimate will be used in the linearization scheme of the inverse problem for the fractional semilinear equation. The result of this section is motivated by \cite{ros2015nonlocal} in which the author considers elliptic integro-differential operators.

\subsection{Comparison principle}
We begin by proving the maximum principle for the fractional Schr\"{o}dinger equation. The definition of weak solutions is stated as follows. %Before doing so, we offer the definition of weak solutions. \\
%	\textcolor{red}{$\checkmark$What are spaces for f and g in proposition? (The space of f and g are not important here. For consistency, I posed it as $L^\infty$ in proposition 3.1)})
\begin{definition}
	The function $u\in H^s (\mathbb R^n)$ is called a weak solution of the fractional Schr\"{o}dinger equation $(-\Delta)^su+au=f$ in $\Omega$ with $u=g$ in $\Omega_e$ if 
	\begin{align*}
	\int_{\mathbb R^n }(-\Delta)^{s/2} u\cdot(-\Delta)^{s/2}\phi dx+
	\int_{\Omega} au\phi dx=\int_\Omega f\phi dx 
	\end{align*}
	with $u-g\in \widetilde H^s(\Omega)$ for any $\phi\in C^\infty_c(\Omega)$. Here $\widetilde H^s(\Omega)$ is the closure of $C^\infty _c(\Omega)$ in $H^s(\mathbb R^n)$.
\end{definition}

The comparison principle can be derived directly from the following maximum principle. 
\begin{proposition}[Maximum principle]\label{max_p}
Let $\Omega\subset\mathbb{R}^{n}$ be a bounded
domain and $a(x)\in L^{\infty}(\Omega)$ be a nonnegative potential. Let
$u\in H^{s}(\mathbb{R}^{n})$ be a weak solution of 
\begin{equation}
\left\{\begin{array}{rl}
(-\Delta)^{s}u+a(x)u=f & \mbox{ in }\Omega,\\
u=g & \mbox{ in }\Omega_{e}.
\end{array}\right.\label{Linear fractional Schrodinger equation}
\end{equation}
Suppose $0\leq f\in L^\infty(\Omega)$ in $\Omega$ and $0\leq g\in L^\infty(\Omega_e)$ in $\Omega_{e}$. Then
$u\geq0$ in $\Omega$.
\end{proposition}

\begin{proof}
	If $u\in H^{s}(\mathbb{R}^{n})$ is a weak solution of \eqref{Linear fractional Schrodinger equation},
	by the weak formulation, we have 
	\begin{equation}
		\int_{\mathbb{R}^{n}}(-\Delta)^{s/2}u\cdot(-\Delta)^{s/2}\phi dx+\int_{\Omega}au\phi dx=\int_{\Omega}f\phi dx,\label{weak forumlation}
	\end{equation}
	for any $\phi\in H_{0}^{s}(\Omega)$. Note that 
	\begin{align*}
		\int_{\mathbb{R}^{n}}(-\Delta)^{s/2}u\cdot(-\Delta)^{s/2}\phi dx 
		& = \iint_{\mathbb{R}^{2n}}\dfrac{(u(x)-u(z))(\phi(x)-\phi(z))}{|x-z|^{n+2s}}dxdz\\
		& =   \iint_{\mathbb{R}^{2n}\backslash (\Omega_{e}\times\Omega_e ) }\dfrac{(u(x)-u(z))(\phi(x)-\phi(z))}{|x-z|^{n+2s}}dxdz,
	\end{align*}
	where we have used $\phi\equiv0$ in $\Omega_{e}$. 
	
	Next, we write $u=u^{+}-u^{-}$ in $\Omega$, where $u^{+}=\max\{u,0\}\chi_{\Omega}$
	and $u^{-}=\max\{-u,0\}\chi_{\Omega}$, with 
	$$\chi_{\Omega}=\begin{cases}
	1 & \mbox{ for }x\in\Omega\\
	0 & \mbox{ otherwise}
	\end{cases}$$
	being the characteristic function. Notice that $u^{+},u^{-}\in H^{s}(\mathbb{R}^{n})$
	due to $u\in H^{s}(\mathbb{R}^{n})$. Recall $u=g\geq0$ in $\Omega_{e}$, then we can take $\phi:=u^{-}\in H_{0}^{s}(\Omega)$
	as a test function. We assume that $u^{-}$ is not identically zero,
	and we want to prove that it will lead to a contradiction.
	
	Since $f\geq0$ and $\phi=u^{-}\geq0$, the right hand side of \eqref{weak forumlation}
	\begin{equation}
		\int_{\Omega}f\phi dx\geq0.\label{nonnegative source}
	\end{equation}
	On the other hand, one has
	\begin{align*}
		& \iint_{\mathbb{R}^{2n}\backslash(\Omega_{e}\times\Omega_e)}\dfrac{(u(x)-u(z))(\phi(x)-\phi(z))}{|x-z|^{n+2s}}dxdz\\
		= & \iint_{\Omega\times\Omega}\dfrac{(u(x)-u(z))(u^{-}(x)-u^{-}(z))}{|x-z|^{n+2s}}dxdz\\
		& +2\int_{\Omega}\int_{\Omega_{e}}\dfrac{(u(x)-g(z))u^{-}(x)}{|x-z|^{n+2s}}dzdx\\
		= & I+II, 
	\end{align*}
	where 
	\begin{align*}
		I & :=   \iint_{\Omega\times\Omega}\dfrac{(u(x)-u(z))(u^{-}(x)-u^{-}(z))}{|x-z|^{n+2s}}dxdz,\\
		II & :=   2\int_{\Omega}\int_{\Omega_{e}}\dfrac{(u(x)-g(z))u^{-}(x)}{|x-z|^{n+2s}}dzdx.
	\end{align*}

	To estimate $I$, since $(u^{+}(x)-u^{+}(z))(u^{-}(x)-u^{-}(z))\leq0$, we obtain 
	\begin{align}
		I %=& \iint_{\Omega\times\Omega}\dfrac{(u(x)-u(z))(u^{-}(x)-u^{-}(z))}{|x-z|^{n+2s}}dxdz\nonumber \\
		\leq & -\iint_{\Omega\times\Omega}\dfrac{(u^{-}(x)-u^{-}(z))^{2}}{|x-z|^{n+2s}}dxdz<0.\label{negative inequality}
	\end{align}
	The last strict inequality holds because $u^{-}$ can not be a constant
	in $\Omega$. If $u^{-}$ is a constant, which means $u\equiv-c_{0}$
	is a negative constant in $\Omega$ (for some constant $c_{0}>0$).
	By the definition of the fractional Laplacian, one can see that for
	$x\in\Omega$, 
	\begin{align*}
		(-\Delta)^{s}u(x) & =  c_{n,s}\mbox{P.V.}\int_{\mathbb{R}^{n}}\dfrac{-c_{0}-u(z)}{|x-z|^{n+2s}}dz\\
		& =   c_{n,s}\int_{\Omega_{e}}\dfrac{-c_{0}-g(z)}{|x-z|^{n+2s}}dz<0,
	\end{align*}
	since $g(z)\geq0$ for $z\in\Omega_{e}$. Therefore, by using \eqref{Linear fractional Schrodinger equation}
	and $a\geq0$ in $\Omega$, we know that 
	\[
	0\leq f=(-\Delta)^{s}u+au<0\mbox{ in }\Omega,
	\]
	which leads to a contradiction. Hence, $u^{-}$ can not be a constant.
	
	For $II$, since $g(z)\geq0$ in $\Omega_{e}$ and $u(x)u^-(x)\leq 0$ in $\Omega$, we deduce that
	$II\leq0$. Therefore, 
	\[
	\iint_{\mathbb{R}^{2n}\setminus(\Omega_{e}\times \Omega_e) }\dfrac{(u(x)-u(z))(\phi(x)-\phi(z))}{|x-z|^{n+2s}}dxdz<0.
	\]
	which contradicts to \eqref{weak forumlation} (because $f\geq0$
	in $\Omega$ and $g\geq0$ in $\Omega_{e}$).
\end{proof}

With the maximum principle, the comparison principle for the fractional Schr\"{o}dinger equation follows immediately. 

\begin{corollary}[Comparison principle]\label{comparison}  
	Let $u_1$ and $u_2$ be weak solutions of 
\begin{equation*}
\left\{\begin{array}{rl}
(-\Delta)^{s}u_1+a(x)u_1=f_1 & \mbox{ in }\Omega,\\
u_1=g_1 & \mbox{ in }\Omega_{e},
\end{array}\right.\ and\  
\left\{\begin{array}{rl}
(-\Delta)^{s}u_2+a(x)u_2=f_2 & \mbox{ in }\Omega,\\
u_2=g_2 & \mbox{ in }\Omega_{e},
\end{array}\right.
\end{equation*}
respectively.
Suppose that $f_1 \geq f_2$ in $\Omega$ and $g_1 \geq g_2$ in $\Omega_e$. Then $u_1\geq u_2$ in $\Omega$.
\end{corollary}
\begin{proof}
	Let $u:=u_1-u_2$ and apply proposition \ref{max_p}, then we complete the proof. Furthermore, one can conclude that $u_1\geq u_2$ in $\mathbb R^n$.
\end{proof}
\begin{remark}
From the above comparison principle, once the solution exists, the uniqueness will automatically hold for the fractional linear Schr\"{o}dinger equation \eqref{Linear fractional Schrodinger equation}.
\end{remark}

\subsection{$L^\infty$ bounds for solutions}
The main goal of this section is stated as follows.
\begin{proposition}\label{prop of sup bound}
Suppose $f\in L^\infty(\Omega)$ and $g\in L^\infty(\Omega_e)$. Let $u$ be a solution to \eqref{Linear fractional Schrodinger equation}, then the following $L^\infty$ estimate 
\begin{align}\label{sup norm estimate}
\|u\|_{L^\infty(\Omega)}\leq \|g\|_{L^\infty (\Omega_e)}+C\|f\|_{L^\infty(\Omega)},
\end{align}
holds for some constant $C>0$ independent of $u,\ f,$ and $g$. 
\end{proposition}
In order to derive \eqref{sup norm estimate}, we need to construct a barrier function for the fractional Schr\"{o}dinger equation.
\begin{lemma}[Barrier]\label{Barrier Lemma}  
Let $\Omega$ be a bounded Lipschitz domain in $\mathbb R^n$ and $a(x)\in L^\infty(\Omega)$ is a nonnegative potential. Then there exists a function $\varphi\in C^\infty _c (\mathbb R^n)$ such that 
\begin{equation}\label{varphi}
\left\{\begin{array}{rl}
(-\Delta)^s \varphi + a(x)\varphi \geq 1 &\mbox{ in }\Omega,\\
\varphi \geq 0 &\mbox{ in } \mathbb R^n,\\
\varphi \leq C & \mbox{ in }\Omega,
\end{array}\right.
\end{equation} 
where $C>0$ is a constant depending on $n,\ s,$ and $\Omega$.
\end{lemma}
\begin{proof}
Let $B_R$ be an arbitrarily large ball such that $\Omega \Subset B_R$ and $\eta\in C^\infty_c(B_R)$ be a smooth cutoff function such that  
\begin{equation*}
0\leq\eta \leq 1\mbox{ in }\mathbb R^n,\    \eta\equiv 1 \mbox{ in }\Omega .
\end{equation*} 
For any $x\in \Omega$, it is clear that $\eta(x)=1$ that is the maximum value of $\eta$. Thus, one has
\begin{align}\label{bdd_eta}
2\eta(x)-\eta(x+y)-\eta(x-y)\geq \eta (x)-\eta (x+y)\geq 0.
\end{align} Recall that for any function $u$ in the Schwartz space, we can also represent the fractional Laplacian as (see \cite{di2012hitchhiks} for instance)
\begin{align}\notag
(-\Delta)^su(x)=\dfrac{c_{n,s}}{2} \int _{\mathbb R^n}\dfrac{2u(x)-u(x+y)-u(x-y)}{|y|^{n+2s}}dy\ \hbox{ for all $x\in\R^n$} ,
\end{align}
where $c_{n,s}$ is the constant in \eqref{c(n,s) constant}. 
For any $x\in\Omega$, from \eqref{bdd_eta} and by using the change of variables $z=x+y$, one has
\begin{align*}
(-\Delta)^s \eta +a(x)\eta &\geq {1\over 2} c_{n,s}\int_{\mathbb R^n}\dfrac{\eta(x)-\eta (z)}{|x-z|^{n+2s}}dz+a(x)\eta \notag\\
&\geq {1\over 2}c_{n,s} \int _{\mathbb R^n\backslash B_R}\dfrac{1}{|x-z|^{n+2s}}dz \notag\\
&\geq \lambda 
\end{align*}
for some constant $\lambda>0$. Now, we let $\varphi(x) := \eta(x)/\lambda $, then we complete the proof.
\end{proof}

It remains to prove the $L^\infty$ bound for the solution $u$.

\begin{proof}[Proof of Proposition \ref{prop of sup bound}]
Let
\begin{equation}\notag
v(x):=\|g\|_{L^\infty(\Omega_e)}+\|f\|_{L^\infty (\Omega)}\varphi (x),
\end{equation} 
where $\varphi(x)$ is the barrier given by Lemma \ref{Barrier Lemma}. From $a(x)\geq 0$ and \eqref{varphi}, we deduce that 
$$
(-\Delta)^s u+a(x)u=f\leq (-\Delta)^s v+a(x)v \mbox{ in }\Omega 
$$
and 
$g\leq v \mbox{ in }\Omega_e$. Applying the comparison principle in Corollary \ref{comparison}, we obtain 
\begin{equation}\notag
u\leq \|g\|_{L^\infty (\Omega_e)}+C\|f\|_{L^\infty(\Omega)} \mbox{ in }\Omega,
\end{equation}
where we use $\varphi \leq C$ in $\Omega$.
Similarly, the same argument will hold for $-u$, which finishes the proof.

\end{proof}

\section{Proof of Theorem \ref{MAIN THEOREM}}
 
In this section, we apply the linearization scheme to transfer the inverse problem for the nonlocal semilinear Schr\"odinger equation to the inverse problem for the nonlocal linear equation.

%To study the inverse boundary value problem for the semilinear equation  $(-\Delta)^{s}+q(x,u)$, the strategy is based on the linearization scheme for this inverse problem. After linearization, the inverse problem for the nonlocal semilinear Schr\"odinger equation can be resolved by solving the problem for the nonlocal linear equation.

\subsection{Linearization}
This linearization method was used in the local type inverse problem, see \cite{victor01, isakov1994global,sun2004inverse,sun2010inverse}. 
\begin{theorem}\label{Linearization Theorem}
Let $n\geq 2$ and $0<s<1$. Let $g$ and $h$ be in $C^3_0(\Omega_{e})$ and $\eta$ be in $\R$. Suppose that $u_{g+\eta h}$ is the solution of \eqref{MAIN THEOREM} with $u_{g+\eta h}=g+\eta h$ in $\Omega_e$. Suppose that $u^{*}$
is the unique solution of the linearized equation 
\begin{equation}
\left\{\begin{array}{rl}
(-\Delta)^{s}u^{*}+\p_tq(x,u_{g})u^{*}=0 & \mbox{ in }\Omega,\\
u^{*}=h & \mbox{ in }\Omega_{e},
\end{array}\right.\label{Linearized equation}
\end{equation}
then we have 
\[
\lim_{\eta\to0}\left\Vert \dfrac{u_{g+\eta h}-u_{g}}{\eta}-u^{*}\right\Vert _{H^{s}(\mathbb{R}^{n})}=0.
\]
\end{theorem}
% % % % % % % % %
%We also have $\lim_{\eta\to0}\left\Vert (-\Delta)^s\LC\dfrac{u_{g+\eta h}-u_{g}}{\eta}\RC-(-\Delta)^s u^{*}\right\Vert _{H^{-s}(\mathbb{R}^{n})}=0$. 
%
% % % % % % % % %

\begin{proof}
First, by using \eqref{Condition for linearized equation}, i.e., $0\leq\partial_tq(x,t)\leq M_0$ for $(x,t)\in \overline\Omega\times\R $  %\footnote{In (1.6), $0<\p_tq\leq M_0/b_0$ only holds in $x\in\overline{\Omega}$ for $|t|\geq r$. We might need to enhance the assumption in (1.6) in order to get this.} 
gives the well-posedness of 
the fractional linear Schr\"{o}dinger equation, i.e., for
a given function $h\in C^3_0(\Omega_{e})\subset H^{s}(\Omega_{e})$,
there exists a unique weak solution $u^{*}\in H^{s}(\mathbb{R}^{n})$
such that $u^*$ solves \eqref{Linearized equation}. 

Next, consider 
\[
w:=\dfrac{u_{g+\eta h}-u_{g}}{\eta}\in H^{s}(\mathbb{R}^{n}),
\]
for $\eta\in\mathbb{R}$, then $w$ is a solution of 
\[
\left\{ \begin{array}{rl}
(-\Delta)^{s}w+Q(x)w=0 & \mbox{ in }\Omega,\\
w=h & \mbox{ in }\Omega_{e},
\end{array}\right.
\]
where 
\[
Q(x)=\int_{0}^{1}\partial_{t}q(x,t u_{g+\eta h}(x)+(1-t)u_{g}(x))dt\geq 0.
\]
Let $v:=w-u^*$, then $v$ solves
\begin{equation}
\left\{ \begin{array}{rl}
(-\Delta)^sv+Q(x)v=-\LC Q(x)-\p_t q (x,u_g)\RC u^* & \mbox{ in }\Omega,\\
v=0 & \mbox{ in }\Omega_{e}.
\end{array}\right.\label{reflected equation}
\end{equation}
By multiplying $v$ on both sides of \eqref{reflected equation}, we can see that  
\begin{align}\label{equ-v}
   \|v\|_{H^{s}(\mathbb{R}^{n})} &\leq C\|\LC Q(x)-\p_t q (x,u_g) \RC  u^*\|_{L^2(\mathbb{R}^n)}\notag\\
&\leq C\|Q(x)-\p_t q (x,u_g)\|_{L^\infty(\Omega)}\|u^*\|_{L^2(\mathbb R^n)},
\end{align}
for some constant $C>0$ independent of $v$ and $u^*$. 

Now, since $u_{g+\eta h}-u_g \in H^s(\mathbb R^n)$ solves 
\begin{align}\notag
\left\{ \begin{array}{rl}
(-\Delta)^s(u_{g+\eta h}-u_g)+Q(x)(u_{g+\eta h}-u_g)=0& \mbox{ in }\Omega,\\
u_{g+\eta h}-u_g=\eta h & \mbox{ in }\Omega_e,
\end{array}\right.
\end{align}
where $h \in C^3_0(\Omega_e)\subset L^\infty(\Omega_e)$. By the $L^\infty$ estimate \eqref{sup norm estimate}, we have 
\begin{equation}\notag
	\|u_{g+\eta h}-u_g\|_{L^\infty (\Omega )}\leq \eta\|h\|_{L^\infty(\Omega_e)}\to 0,
\end{equation}
as $\eta\to 0$.

By the continuity of $\partial_tq(x,t)$, it implies that
\begin{align*}
\|\p_t q (x,t u_{g+\eta f}+(1-t)u_g)-\p_t q  (x,u_g)\|_{L^\infty(\Omega)}\to 0
\end{align*}
as $\eta \to 0$  %, where we have used the condition that $q_t(x,t)$ is continuous in $t$. 
that leads to
\begin{align*}
\|Q(x)-\p_t q (x,u_g)\|_{L^\infty(\Omega)} \leq\int_0^1\|\p_t q (x,t u_{g+\eta f}+(1-t)u_g)-\p_t q  (x,u_g)\|_{L^\infty(\Omega)}d\zeta \to 0,
\end{align*}
whenever $\eta \to 0$. Combining with \eqref{equ-v}, the proof is complete.

\end{proof}

\subsection{Proof of Theorem \ref{MAIN THEOREM}}
Now, we are ready to prove our main theorem.

\begin{proof}[Proof of Theorem \ref{MAIN THEOREM}]
	For any $g\in C^3_0(W)\subset C^3_0(\Omega_{e})$,
	let $u_{g}^{(1)}\in H^{s}(\mathbb{R}^{n})$ be a solution of 
	\[
	\left\{\begin{array}{rl}
	(-\Delta)^{s}u_{g}^{(1)}+q_1(x,u_{g}^{(1)})=0 & \mbox{ in }\Omega,\\
	u_{g}^{(1)}=g & \mbox{ in }\Omega_{e}.
	\end{array}\right.
	\]
	From the Cauchy data assumption \eqref{Partial Cauchy data}, there exists a solution
	$u_{g}^{(2)}\in H^{s}(\mathbb{R}^{n})$ such that $u_{g}^{(2)}=u_{g}^{(1)}$,
	$\mathcal{N}^s_{q_2}u_{g}^{(2)}=\mathcal{N}^s_{q_1}u_{g}^{(1)}$ in $W$
	and $u_{g}^{(2)}$ solves 
	\[
	\left\{\begin{array}{rl}
	(-\Delta)^{s}u_{g}^{(2)}+q_2(x,u_{g}^{(2)})=0 & \mbox{ in }\Omega,\\
	u_{g}^{(2)}=g & \mbox{ in }\Omega_{e}.
	\end{array}\right.
	\]
	For any $h\in C^3_0(W)\subset C^3_0(\Omega_e)$, % and $\eta\in\R$,
	using the Cauchy data assumption again, there are solutions $u_{g+\eta h}^{(j)}\in H^{s}(\mathbb{R}^{n})$
	of 
	\begin{align*} 
	\left\{\begin{array}{rl}
		(-\Delta)^{s}u_{g+\eta h}^{(j)}+q_j(x,u_{g+\eta h}^{(j)})=0 & \mbox{ in }\Omega,\\ u_{g+\eta h}^{(j)}=g+\eta h & \mbox{ in } \Omega_{e},
			\end{array}\right.
	\end{align*}
		 for
	$j=1,2$, such that 
	\begin{align}\notag
	\mathcal{N}^s_{q_1}u_{g+\eta h}^{(1)}=\mathcal{N}^s_{q_2}u_{g+\eta h}^{(2)}\mbox{ in }W\subset \Omega_{e}  \label{perturbed solutions}
	\end{align}
	for any $\eta\in\mathbb{R}$.
	We differentiate the above equation with respect to $\eta$ at $\eta=0$ and use \eqref{relation of Ns and DN map} with Theorem \ref{Linearization Theorem}, then we obtain that 
	\begin{equation}
	\mathcal{N}^s_{q_1}\dot{u}_{g,h}^{(1)}=\mathcal{N}^s_{q_2}\dot{u}_{g,h}^{(2)}\mbox{ in }W,\label{equal Neumann data}
	\end{equation}
	where $\dot{u}_{g,h}^{(1)}$ and $\dot{u}_{g,h}^{(2)}$ are solutions
	of 
	\begin{align}\label{linear 1}
		(-\Delta)^{s}\dot{u}_{g,h}^{(1)}+\p_t q_1 (x,u_{g}^{(1)})\dot{u}_{g,h}^{(1)}=0\mbox{ in }\Omega\mbox{ with }\dot{u}_{g,h}^{(1)}=h\mbox{ in }\Omega_{e},
	\end{align}
	and 
	\begin{align}\label{linear 2}
	(-\Delta)^{s}\dot{u}_{g,h}^{(2)}+\p_t q_2 (x,u_{g}^{(2)})\dot{u}_{g,h}^{(2)}=0\mbox{ in }\Omega\mbox{ with }\dot{u}_{g,h}^{(2)}=h\mbox{ in }\Omega_{e}.
	\end{align}
	
	Recall that $\partial_tq_j\in L^{\infty}$ and $\partial_tq_j\geq0$ in $\Omega$. Thus, by
	using the well-posedness of the linearized fracional Schr\"{o}dinger equation
	(see Section 2 in \cite{ghosh2016calder}), nonlocal DN maps $\Lambda_{\partial_tq_j}$
	exist and are defined by
	\[
	\Lambda_{\partial_tq_j(x,u_{g}^{(j)})}:\dot{u}_{g,h}^{(j)}|_{\Omega_{e}}\to(-\Delta)^{s}\dot{u}_{g,h}^{(j)}|_{\Omega_{e}},\mbox{ for }j=1,2.
	\]
	For a fixed $g\in C^3_0(W)$, since \eqref{equal Neumann data} holds for any $h\in C^3_0(W)$, by using \eqref{relation of Ns and DN map}, we derive 
	\begin{align}\notag
	\Lambda_{\partial_tq_1(x,u^{(1)}_g)}h|_W=\Lambda_{\partial_t q_2(x,u^{(2)}_g)}h|_W\mbox{ for any }h\in C^3_0(W).
	\end{align}
	Thus, we can use global uniqueness of the fractional linear Schr\"{o}dinger equation (see Theorem 1.1 in \cite{ghosh2016calder}), then we can conclude 
	\begin{equation}\label{equal potentials}
		\p_t q_1 (x,u^{(1)}_g)=\p_t q_2(x,u^{(2)}_g)\mbox{ in }\Omega.
	\end{equation}
	Via \eqref{equal potentials}, we know that $\dot{u}_{g,h}^{(1)}$ and $\dot{u}_{g,h}^{(2)}$ solve \eqref{linear 1} and \eqref{linear 2} (with the same coefficients), respectively. From the well-posedness for the fractional linear  Schr\"{o}dinger equation again, one can see that the weak solution will be unique, that is,
	\begin{equation*}\label{whole equality}
	\dot{u}_{g,h}^{(1)}=\dot{u}_{g,h}^{(2)} \mbox{ in } H^s (\mathbb R^n).
	\end{equation*}
	In particular, we take the original $g$ by $\eta g$ and $h$ by $g$, then we have 
	\begin{equation*}
	\dot{u}^{(1)}_{\eta g,g}=\dot{u}^{(2)}_{\eta g,g}\mbox{ in }H^s(\mathbb R^n), \mbox{ for any }\eta \in \mathbb R.
	\end{equation*}
	By using Theorem \ref{Linearization Theorem}, this implies that 
	\begin{equation*}
	\dfrac{d}{d\eta}u^{(1)}_{\eta g}=\dfrac{d}{d\eta}u^{(2)}_{\eta g}\mbox{ in }H^s (\mathbb R^n )\mbox{ for any }\eta \in \mathbb R .
	\end{equation*}
	Hence, there exists a function $\psi=\psi(x)\in H^s(\mathbb R^n)$ independent of $\eta$ such that 
	\begin{equation}\label{equivalence class}
	u^{(1)}_g=u^{(2)}_g+\psi \mbox{ in } H^s(\mathbb R^n) 
	\end{equation}
	for all $g\in C^3_0(\Omega_e)$. 
	%Note that from the above construction, then one can acquire that the function $\psi$ is independent of $g$, which means \eqref{equivalence class} holds for all $g\in H^{2s}_0 (\Omega_e)\cap L^\infty(\Omega_e)$. 
	
%	----------------------------------------------
	
%	{\color{red}
%	Changes:\\
%	\begin{enumerate}
%	\item I change $\mathcal{N}_s$ to $\mathcal{N}^s_{q}$ which can clearly express its dependence on potential $q$
%	\end{enumerate}
%	Questions:\\
%	\begin{enumerate}
%	   \item $\mathcal{N}^s_{q_1} u^1=\mathcal{N}^s_{q_2} u^2$ %implies $\Lambda_{\p_t q_1} f= \Lambda_{\p_t q_2} f$? Yes.
	   
%	   \item Here $\psi$ is formally like $	u^{(1)}_0-	u^{(2)}_0$ by fundamental theorem of calculus. Then by Cauchy data, we conclude in the following that $\psi=0$. This means we have the same solution $u^{(1)}_0=	u^{(2)}_0$ for nonlinear Schr\"odinger equation for different potential $q_j(x,u(x))$ with external data $0$. It is quite weird! Note that it might be not a issue, I just feel it is too strong.
%	\end{enumerate}
%	}
%		----------------------------------------------
		
	Now, by using the assumption on Cauchy data \eqref{Partial Cauchy data}, 
	\begin{equation}\notag
	\{(u^{(1)}_g|_W,\mathcal N^s_{q_1} u^{(1)}_g|_W) \}=\{(u^{(2)}_g|_W,\mathcal N^s_{q_2} u^{(2)}_g|_W) \},	
	\end{equation}
	we can obtain $\psi \in H^s(\mathbb R ^n)$ such that
	\begin{equation}\label{zero exterior data}
	\psi=(-\Delta)^s\psi =0\mbox{ in }W\subset \Omega_e.
	\end{equation}
	We apply Theorem 1.2 in \cite{ghosh2016calder}, then the function $\psi \equiv 0$ in $\mathbb R^n$. Finally, we substitute $u^{(1)}_g(x)=u^{(2)}_g(x)$ for $x\in \mathbb R^n$ into the fractional semilinear Schr\"{o}dinger equation % to get
	\begin{equation}\notag
	(-\Delta)^su^{(1)}_g+q_1(x,u^{(1)}_g)=(-\Delta)^su^{(2)}_g+q_2(x,u^{(2)}_g)\mbox{ in }\Omega,
	\end{equation} 
	which implies   
	\begin{equation}\notag
    q_1(x,u)=q_2(x,u)\mbox{ for }x\in\Omega.		
	\end{equation}
	This completes the proof of our main result.
\end{proof}
\vskip1cm
\textbf{Acknowledgment.}
The second author was supported in part by MOST of Taiwan 160-2917-I-564-048. Both authors would like to thank Camelia Pop for helpful discussions.

\bibliographystyle{plain}
\bibliography{ref}

\begin{thebibliography}{10}

\bibitem{bisci2016variational}
Giovanni~Molica Bisci, Vicentiu~D Radulescu, and Raffaella Servadei.
\newblock {\em Variational methods for nonlocal fractional problems}, volume
  162.
\newblock Cambridge University Press, 2016.

\bibitem{di2012hitchhiks}
Eleonora Di~Nezza, Giampiero Palatucci, and Enrico Valdinoci.
\newblock Hitchhiker's guide to the fractional {S}obolev spaces.
\newblock {\em Bulletin des Sciences Math{\'e}matiques}, 136(5):521--573, 2012.

\bibitem{dipierro2014nonlocal}
Serena Dipierro, Xavier Ros-Oton, and Enrico Valdinoci.
\newblock Nonlocal problems with {N}eumann boundary conditions.
\newblock {\em Rev. Mat. Iberoam.}, 33:377--416, 2017.

\bibitem{frank2016uniqueness}
Rupert Frank, Enno Lenzmann, and Luis Silvestre.
\newblock Uniqueness of radial solutions for the fractional {L}aplacian.
\newblock {\em Communications on Pure and Applied Mathematics},
  69(9):1671--1726, 2016.

\bibitem{ghosh2017calder}
Tuhin Ghosh, Yi-Hsuan Lin, and Jingni Xiao.
\newblock The {C}alder\'{o}n problem for variable coefficients nonlocal
  elliptic operators.
\newblock {\em arXiv preprint arXiv:1708.00654}, 2017.

\bibitem{ghosh2016calder}
Tuhin Ghosh, Mikko Salo, and Gunther Uhlmann.
\newblock The {C}alder{\'o}n problem for the fractional {S}chr{\"o}dinger
  equation.
\newblock {\em arXiv:1609.09248}, 2016.

\bibitem{victor01}
Victor Isakov.
\newblock Uniqueness of recovery of some quasilinear partial differential
  equations.
\newblock {\em Commun. in partial differential equations}, 26(11,
  12):1947--1973, 2001.

\bibitem{VictorN}
Victor Isakov and A~Nachman.
\newblock Global uniqueness for a two-dimensional elliptic inverse problem.
\newblock {\em Trans.of AMS}, 347:3375--3391, 1995.

\bibitem{isakov1994global}
Victor Isakov and John Sylvester.
\newblock Global uniqueness for a semilinear elliptic inverse problem.
\newblock {\em Communications on Pure and Applied Mathematics},
  47(10):1403--1410, 1994.

\bibitem{Kay}
Kay Kirkpatrick and Yanzhi Zhang.
\newblock Fractional {S}chr\"odinger dynamics and decoherence.
\newblock {\em Physica D: Nonlinear Phenomena}, 332(14):41--54, 2016.

\bibitem{ros2015nonlocal}
Xavier Ros-Oton.
\newblock Nonlocal elliptic equations in bounded domains: a survey.
\newblock {\em Publicacions Matem\`{e}tiques.}, 60(1):3--26, 2016.

\bibitem{ruland2017fractional}
Angkana R{\"u}land and Mikko Salo.
\newblock The fractional {C}alder\'{o}n problem: low regularity and stability.
\newblock {\em arXiv preprint arXiv:1708.06294}, 2017.

\bibitem{sun2004inverse}
Ziqi Sun.
\newblock Inverse boundary value problems for a class of semilinear elliptic
  equations.
\newblock {\em Advances in Applied Mathematics}, 32(4):791--800, 2004.

\bibitem{sun2010inverse}
Ziqi Sun.
\newblock An inverse boundary-value problem for semilinear elliptic equations.
\newblock {\em Electronic Journal of Differential Equations (EJDE)[electronic
  only]}, 37:1--5, 2010.

\bibitem{uhlmann2009calderon}
Gunther Uhlmann.
\newblock {E}lectrical impedance tomography and {C}alder\'on's problem.
\newblock {\em Inverse Problems}, 25:123011, 2009.

\bibitem{Uzar}
Neslihan Uzar and Sedat Ballikaya.
\newblock Investigation of classical and fractional {B}ose-{E}instein
  condensation for harmonic potential.
\newblock {\em Physica A: Statistical Mechanics and its Applications},
  392(8):1733--1741, 2013.

\end{thebibliography}

\end{document}